\documentclass[11pt,a4paper]{article}

\usepackage{amsmath}
\usepackage{amssymb}
\usepackage{ntheorem} 
\usepackage{hyperref} 
\usepackage{verbatim}
\usepackage{authblk}
\usepackage{xcolor}

\allowdisplaybreaks

\theorembodyfont{\slshape}

\newtheorem{lem}{Lemma}[section]

\newtheorem{prop}[lem]{Proposition}
\newtheorem{cor}[lem]{Corollary}

\theorembodyfont{\rm}   

\newtheorem{rem}[lem]{Remark}
\newtheorem{ex}[lem]{Example}
\newtheorem{defn}[lem]{Definition}

\newtheorem*{mthmn}{Main Theorem.}
\newtheorem*{quest1}{Question 1.}
\newtheorem*{quest2}{Question 2.}

\newenvironment{proof}{\noindent\textit{Proof. }}{\hfill $\Box$\linebreak}

\DeclareMathSymbol{\ordinaryl}{\mathalpha}{letters}{`l}     
\begingroup                            
\lccode`\~=`\l
\lowercase{\gdef ~{\ifnum\the\mathgroup=-1 \ell \else \ordinaryl \fi}}
\endgroup
\mathcode `\l="8000

\newcommand{\D}{\displaystyle}

\newcommand{\BB}{\mathbb{B}}
\newcommand{\CC}{\mathbb{C}}
\newcommand{\NN}{\mathbb{N}}
\newcommand{\PP}{\mathbb{P}}
\newcommand{\RR}{\mathbb{R}}
\renewcommand{\SS}{\mathbb{S}}

\newcommand{\Ccal}{\mathcal{C}}
\newcommand{\Fcal}{\mathcal{F}}
\newcommand{\Mcal}{\mathcal{M}}
\newcommand{\Ncal}{\mathcal{N}}
\newcommand{\Pcal}{\mathcal{P}}

\newcommand{\nuc}{\mathfrak{n}}

\newcommand{\Repa}[1]{\operatorname{Re}(#1)}
\newcommand{\repa}[1]{\operatorname{Re}(#1)}

\newcommand{\impa}[1]{\operatorname{Im}(#1)}
\newcommand{\Int}{\hbox{Int}}

\newcommand{\eps}{\varepsilon}

\newcommand{\qand}{\quad \hbox{and} \quad}

\date{}

\author{Thomas Pawlaschyk and Nikolay Shcherbina}

\title{On compact sets possessing \textit{q}-convex functions}

\begin{document}

\maketitle

\begin{abstract}We show that there exists a $q$-convex function in a neighborhood of a compact set $K$ in a complex manifold $\Mcal$ if and only if the $q$-nucleus of this compact set is empty. The latter can be characterized as the maximal $q$-pseudoconcave subset of $K$, i.e., a subset of $K$ containing all other compact $q$-pseudoconcave subsets in $K$.
\end{abstract}

%
%
\renewcommand{\thefootnote}{}

\section{Introduction}

Plurisubharmonic functions are crucial in the study of functions in several complex variables. They originate from (sub-)harmonic functions and were introduced by Lelong and Oka in the early 1940s. In the smooth case, these were generalized to $q$-convex functions by Rothstein~\cite{Ro} and Grauert~\cite{Gr} in the 1950s and, in the upper semi-continuous case, to $q$-plurisubhamonic functions by Hunt and Murray \cite{HM} in 1978. Andreotti and Grauert used them to study cohomology groups of $q$-convex complex manifolds and spaces in the 1960s (cf.~\cite{HL}).
Plurisubharmonicity is strongly related to Kobayashi hyperbolicity of bounded sets, since the existence of strictly plurisubharmonic functions guarantees that the Kobayashi pseudo-metric is actually hyperbolic, i.e., a metric~\cite{Koba}. For this reason, the second named author of this paper examined the question on the existence of strictly plurisubharmonic functions in a open neighborhood of a given compact set $K$ in complex manifolds~\cite{Shc21}. It turned out that their existence is closely related to the emptiness of the \emph{nucleus}, a subset of $K$ which is constituted by the union of all pseudoconcave sets in $K$.
It is natural to ask whether a similar result transfers to $q$-convexity. In this paper we give necessary and sufficient conditions on the existence of strictly $q$-convex functions (with corners) in a neighborhood of a given compact set $K$. For this, we introduce the notion of the \emph{$q$-nucleus} of $K$ which is eventually composed by all $q$-pseudoconcave subsets of $K$, i.e., complements of $q$-pseudoconvex sets in the sense of Rothstein~\cite{Ro}. Our main results states that there exists a $q$-convex function in a neighborhood of $K$ if and only if the $q$-nucleus of $K$ is empty (see Main Theorem).

\section{The \textit{q}-convexity in terms of spherical hats}

For $z\in\CC^n$ we write $\|z\|_\infty:=\max\{|z_j|:1\leq j\leq n\}$ and $\Delta^n:=\{z \in \CC^n : \|z\|_\infty <1\}$. Denote also the unit ball by $\BB_1^n(0):=\{z \in \CC^n : \|z\|<1\}$, where $\|\cdot\|$ is the Euclidean norm in $\CC^n$.

\begin{defn} Fix integer numbers $k,m \geq 1$ and let $0 < r,s < 1$. An Euclidean $(k,m)$-\emph{Hartogs figure} is a set of the form 
\[
H^{k,m}=H^{k,m}_{r,s}:=\{z \in \Delta^k\times\Delta^{m}: \|(z_1,\ldots,z_k)\|_\infty < r \ \mbox{or} \ \|(z_{k+1},\ldots,z_{k+m})\|_\infty > s\}.
\]
\end{defn}

In the following, $\Mcal^n=\Mcal$ always denotes a complex manifold of dimension $n \geq 2$ unless otherwise stated.

\begin{defn} Let $q \in \{1,\ldots,n-1\}$ be an integer number. An open set $\Omega$ in $\Mcal$ is \emph{$q$-pseudoconvex} (in $\Mcal$) if it admits the \emph{Kontinuit\"atssatz} with respect to $(n-q)$-polydiscs, i.e.~for every $(q,n-q)$-Hartogs figure $H^{q,n-q}$ and every injective holomorphic map $\Phi:\Delta^n\to\Mcal$ such that $\Phi(H^{q,n-q}) \subset \Omega$ we have $\Phi(\Delta^n) \subset \Omega$.
\end{defn}

\begin{rem} Notice that the above described $q$-pseudoconvexity is defined in the sense of Rothstein~\cite{Ro}. It is equivalent to the $(n{-}q{-}1)$-pseudoconvexity in the sense of S\l{}odkowski~\cite{Sl} (see~\cite{Paw} for a list of equivalent notions of $q$-pseudoconvexity). In the case $q=n-1$, the $(n-1)$-pseduoconvexity in the sense of Rothstein is simply the classical pseudoconvexity.
\end{rem}

We introduce another notion of generalized pseudoconvexity which is based on half spheres rather than Hartogs figures.

\begin{defn}\label{defn-qpsc-hats} Let $0 < r < 1$ be a real number and fix some integer number $1 \leq k \leq n$. We consider the \emph{spherical hat} defined as
\[
\SS_r^k := \{z=(z_1,\ldots,z_k) \in \CC^k : \|z\|=1 \ \hbox{and}\ \Repa{z_1} \geq r\}
\]
and the \emph{filled (spherical) hat}
\[
\widehat{\SS}_r^k := \{z=(z_1,\ldots,z_k) \in \CC^k : \|z\|\leq 1 \ \hbox{and}\ \Repa{z_1} \geq r\}.
\]
A pair of sets $(S^k,\widehat{S}^k)=(S, \widehat{S})$ in $\Mcal$ is said to be a \emph{spherical hat pair of order $k$} if there exist a real number $0<r<1$, a filled spherical hat $\widehat{\SS}_r^{k}$, an open neighborhood $U$ of $\widehat{\SS}_r^{k} \times \overline{\Delta^{n-k}}$ in $\CC^n$ and an injective holomorphic map $\Phi:U\to\Mcal$ such that 
\[
S=S^k:=\Phi(\SS_r^{k}\times\Delta^{n-k}) \qand \widehat{S}=\widehat{S}^k:=\Phi\big(\widehat{\SS}_r^{k}\times\Delta^{n-k}\big).
\] 
We also denote by $\Int(\widehat{S}^k)$ the image of the interior of $\widehat{\SS}_r^{k}\times\Delta^{n-k}$ by $\Phi$ and call it the \emph{filling of} $S^k$.
\end{defn}

Now we compare the two notions of $q$-pseudoconvexity defined above.

\begin{prop} \label{prop-qpsc-caps} Let $\Omega$ be an open set in $\Mcal$. Then the following properties are equivalent.

\begin{enumerate}

\item \label{prop-qpsc-caps-1}  $\Omega$ is $q$-pseudoconvex.

\item \label{prop-qpsc-caps-4} For every spherical hat pair $(S,\widehat{S})$ of order $n-q+1$ such that $\Omega$ contains the spherical hat $S=S^{n-q+1}$, the set $\Omega$ also contains the filled hat $\widehat{S}=\widehat{S}^{n-q+1}$.


\end{enumerate}

\end{prop}

\begin{proof} \textbf{(\ref{prop-qpsc-caps-1}) implies (\ref{prop-qpsc-caps-4}).} We prove by contradiction. Assume that $\Omega$ fulfills (\ref{prop-qpsc-caps-1}) but not (\ref{prop-qpsc-caps-4}) of the above proposition. Then we can find a spherical hat $\SS_r^{n-q+1}$ and an injective holomorphic map on a neighborhood $U$ in $\CC^n$ of $\widehat{\SS}_r^{n-q+1}\times\Delta^{q-1}$ so that $\Phi(\SS_r^{n-q+1}\times\Delta^{q-1})$ lies in $\Omega$, but the complement of $\Omega$ intersects $\Phi(\Int(\widehat{\SS}_r^{n-q+1}\times\Delta^{q-1}))$. Define $D:=\Phi^{-1}(\Omega)$. Then $D$ is $q$-pseudoconvex in~$U$ and, hence, admits the continuity principle with respect to $(n-q)$-dimensional analytic sets (cf., e.g., Theorem 4.3.2 in \cite{Paw}). Let $p$ be an intersection point of the interior of $\widehat{\SS}_r^{n-q+1}\times\Delta^{q-1}$ and the boundary of $D$ in $U$ with coordinates $p=(p_1,p',p'')$, where $p'=(p_2,\ldots,p_{n-q+1})$ and $p''=(p_{n-q+2},\ldots,p_n)$. Now we define $(n-q)$-dimensional analytic sets as follows:
\[
A_t:=\{(t+i\impa{p_1},z') \in \CC\times\CC^{n-q} : t^2+(\impa{p_1})^2 + \|z'\|^2 < 1 \} \times \{p''\},
\]
with $t_0:= \repa{p_1} \leq t < \sqrt{1-(\impa{p_1})^2 }=:t_1$ and $z'=(z_2,\ldots,z_{n-q+1})$. Observe that

\begin{itemize}

\item $\partial A_t:= \{(t+i\impa{p_1},z') \in \CC\times \CC^{n-q} : t^2+(\impa{p_1})^2 + \|z'\|^2 = 1 \} \times \{p''\}$ lies in $\SS_r^{n-q+1}\times\Delta^{q-1} \subset D$ for every $t \in [t_0,t_1]$,

\item $A_{t^*} \subset D$ for all $t^*<t_1$ close enough to $t_1$, since $\partial A_{t_1}=\{(t_1+i\impa{p_1},p',p'')\} \subset D$,

\item and $p \in A_{t_0} \cap \partial D \neq \emptyset$.

\end{itemize}

Then we can find $t^{**} \in (t_0,t^*)$ such that $A_{t^{**}}$ intersects the boundary of $D$ for the first time as $t \downarrow t_0$. But this contradicts the continuity principle with respect to $(n-q)$-dimensional analytic sets of the $q$-pseudoconvex domain $D$ in $U$. Hence, $\Omega$ has to fulfill property~(\ref{prop-qpsc-caps-4}).

\medskip

\textbf{(\ref{prop-qpsc-caps-4}) implies (\ref{prop-qpsc-caps-1}).} We prove again by contradiction. Assume that $\Omega$ admits property~(\ref{prop-qpsc-caps-4}), but is not $q$-pseudoconvex. Then there is a $(q,n-q)$-Hartogs figure $H^{q,n-q}_{r,s}$ and an injective holomorphic map $\Phi:\Delta^n \to \Mcal$ such that $\Phi(H^{q,n-q}_{r,s}) \subset \Omega$, but $\Phi(\Delta^n)$ intersects the complement of $\Omega$. 

Let $D:=\Phi^{-1}(\Omega)$. Then, by assumption made on the Hartogs figure, $H_{r,s}^{q,n-q} \subset D$, but $(H_{r,s}^{q,n-q})^c \cap D^c \neq \emptyset$ in $\Delta^n$. Now let 
\[
r_0 := \min \{ r>0 : H_{r,s}^{q,n-q} \cap D^c \neq \emptyset \}.
\]
This means that the closure of $H_{r_0,s}^{q,n-q}$ intersects the boundary of $D$ for the first time when the radius $r$ increases towards $1$. This intersection takes place at the boundary part of $H_{r_0,s}^{q,n-q}$ in $\Delta^n$ which lies in $\partial(\Delta^q_r) \times \Delta^{n-q}_s$. Let $p$ be such an intersecting point. Then there is an index $j_1 \in \{1,\ldots,q\}$ such that
\[
p \in \{p^*\}  \times \partial({\Delta}_{r_0}) \times \{p^{**}\} \times \Delta^{n-q}_s
\]
with $p^*:=(p_1,\ldots,p_{j_1-1}) \in \overline{\Delta^{{j_1}-1}_{r_0}}$ and  $p^{**}:=(p_{j_1+1},\ldots,p_{q}) \in \overline{\Delta^{q-{j_1}-1}_{r_0}}$. Consider the set $\Pi:=\{p^*\}\times \CC \times \{p^{**}\} \times \Delta^{n-q}$ and let $D' \subset \Delta^{n-q+1}$ be a domain such that 
\[
\pi\Big(D \cap \Pi \Big) = D',
\]
where $\pi:\CC^n \to \CC^{n-q+1}$ denotes the coordinate projection of the $(z_1,\ldots,z_n)$-coordinates to the $(z_{j_1},z'')=(z_{j_1},z_{q+1},\ldots,z_n)$-coordinates. Observe that
\[
\pi\big( H^{q,n-q}_{r_0,s} \cap \Pi \big) = H' = \{|z_{j_1}| < r_0\} \cup \{\|z''\|_\infty > s\}
\]
is an $(1,n-q)$-Hartogs figure in $\CC^{n-q+1}$ such that
\[
H' \subset D', \quad \text{but}\quad \Delta^{n-q+1} \cap (D')^c \neq \emptyset.
\] 
Now we can use the same technique as in the second part of the proof of Proposition~2.1 in~\cite{Shc21} and construct a spherical hat pair $(S',\hat S')$ of order $n-q+1$ in $\CC^{n-q+1}$ such that 
\[
S' \subset D', \quad \text{but}\quad \hat S' \cap (D')^c \neq \emptyset.
\]
We re-arrange the coordinates such that $z_{j_1}$ becomes $z_1$, and $z''$ becomes $(z_2,\ldots,z_{n-q+1})$. Then
\[
\big(S'\times\{(p^*, p^{**})\}\big) \subset D, \quad \text{but}\quad \big(\hat S'\times\{(p^*,p^{**})\}\big) \cap D^c \neq \emptyset.
\]
Recall that $(p^*,p^{**}) \in \Delta^{q-1}$ and notice that after a small pertubation, if necessary, we can assume that $\big(S'\times\{(p^*, p^{**})\}\big) \Subset D$. Hence, there exists a small $\eps \in (0,r)$ such that
\[
\big(S'\times\Delta_\eps^{q-1}\big) \subset D, \quad \text{but}\quad  \big(\hat S'\times\Delta_\eps^{q-1}\big) \cap D^c \neq \emptyset.
\]
Let $S:=\big(S'\times\Delta_\eps^{q-1}\big)$ and $\hat S:= \big(\hat S'\times\Delta_\eps^{q-1}\big)$. Then we obtain a spherical hat pair $(\Phi(S),\Phi(\hat S))$ of order $n-q+1$ with $\Phi(S) \subset \Omega$, but $\Phi(\hat S) \cap \Omega^c \neq \emptyset$. This contradicts the assumption made on $\Omega$ to fulfill property~(\ref{prop-qpsc-caps-4}) of this proposition. Hence, $\Omega$ must be $q$-pseudoconvex.



\end{proof}

In what follows, we study the relation of $q$-pseudoconvex sets and $q$-convex functions.

\begin{defn} Let $q\in\{1,\ldots,n\}$.

\begin{enumerate}

\item We say that a $\Ccal^2$-smooth function $\psi$ on $\Mcal$ is \emph{weakly $q$-convex} if for every point $p$ the Levi form of $\psi$ at $p$ has at most $q-1$ negative eigenvalues. We say it is \emph{$q$-convex} if for every point $p$ the Levi form of $\psi$ at $p$ has at most $q-1$ non-positive eigenvalues. 

\item Furthermore, $\psi$ is \emph{(weakly) $q$-convex with corners} on $\Mcal$ if for each point $p$ in $\Mcal$ there are a neighborhood $U$ of $p$ and finitely many (weakly) $q$-convex functions $\psi_1,\ldots,\psi_k$ on $\Mcal$ such that $\psi=\max\{\psi_j:j=1,\ldots,k\}$ on $U$.

\item We say that $\Mcal$ is \emph{$q$-complete with corners} if $\Mcal$ admits a $q$-convex with corners exhaustion function $\psi$, i.e., $\{ \psi < c\} \Subset \Mcal$ for every number $c \in \RR$. 

\end{enumerate}

\end{defn}

\begin{rem}\label{rem-qvx-fcts} 1. Counting the eigenvalues of the Levi form, it is immediate that a $q$-convex function is weakly $q$-convex. Furthermore, it is obvious that a weakly $q$-convex function is weakly $q$-convex with corners.

\medskip

2. Hunt and Murray \cite{HM} proved that a $\Ccal^2$-smooth function $\psi$ defined on an open set $U$ in $\CC^n$ is weakly $q$-convex if and only if it is \emph{$(q-1)$-plurisubharmonic} in their sense.

\medskip 

3. If $\Mcal$ is compact, there are no $q$-convex functions with corners on $\Mcal$. Indeed, otherwise, if such a function, say $\psi$, exists, it attains a maximum in $\Mcal$ at some point $p$. Then there is a $q$-convex function $\psi_{j_0}$ in a neighborhood $U$ of $p$ such that $\max\{\psi(z): z \in U\}=\psi_{j_0}(p)$. But all the eigenvalues of the Levi form of $\psi_{j_0}$ at $p$ have to be non-positive. This contradicts the definition of the $q$-convexity of $\psi_{j_0}$.

\end{rem}

We are especially interested in complements of $q$-pseudoconvex sets.

\begin{defn} Let $A$ be a closed set in $\Mcal$. Then $A$ is called \emph{$q$-pseudoconcave (in $\Mcal$)} if $\Mcal\setminus A$ is $q$-pseudoconvex.
\end{defn}

The following result is due to S{\l}odkowski~\cite[Proposition~5.2]{Sl}. We give it here as an example for $q$-pseudoconcave sets.

\begin{ex}\label{ex-q-pscve-analytic} If $A$ is an analytic subset of $\CC^n$ such that its dimension is at least $q$ at each point $z \in A$, then $A$ is $q$-pseudoconcave in $\CC^n$. 
\end{ex}

The $q$-convex functions admit the local maximum principle on $q$-pseudoconcave sets.

\begin{prop}\label{prop-qcvx-loc-max} Let $A$ be a closed set in $\Mcal$. Then the following properties are equivalent.

\begin{enumerate}

\item $A$ is $q$-pseudoconcave in $\Mcal$.

\item For every $z$ in $A$ there exists an open neighborhood $V$ of $z$ in $\Mcal$ such that $A \cap V$ is  $q$-pseudoconcave in $V$.

\item For every $z$ in $A$ there exists an open neighborhood $U$ of $z$ in $\Mcal$ such that for every  compact set $L \Subset U$ and every weakly $q$-convex function $\psi$ defined in some neighborhood of $L$ we have 
\[
\max\{\psi(z): z \in A \cap L\} = \max\{\psi(z): z \in A \cap \partial L\}.
\]

\end{enumerate}

Notice that we set $\max\{\psi(z): z \in A \cap E\}=-\infty$, if $E$ is some closed set in $\Mcal$ with $A \cap E=\emptyset$.

\end{prop}

\begin{proof} The proof is based on the corresponding proofs of Theorems~4.2 and~5.1 in S{\l}odkowski~\cite{Sl}. It can be found in Proposition~3.3 of~\cite{HST14}. 




\end{proof}

As an application of the local maximum principle, we derive an example for $q$-pseudoconvex sets using $(n-q)$-convex functions.

\begin{prop}\label{prop-q-cvx-psc} Assume that an open set $\Omega$ in $\Mcal$ admits an $(n-q)$-convex exhaustion function. Then $\Omega$ is $q$-pseudoconvex.
\end{prop}

\begin{proof} Let $H^{q,n-q}=H^{q,n-q}(r,s)$ be a Hartogs figure with $\Phi(H^{q,n-q}) \subset \Omega$. Now for $\zeta \in \Delta^q$ and $s < \sigma < 1$ consider the polydiscs 
\[
A_{\zeta, \sigma} := \{\zeta\} \times \Delta_{\sigma}^{n-q}.
\]
Then 
\begin{eqnarray}\label{slices-in-omega}
\Delta^n = H^{q,n-q} \cup \bigcup_{\stackrel{\|\zeta\|_{\infty}<1}{s < \sigma < 1}} A_{\zeta,\sigma},
\qand \partial A_{\zeta,\sigma} := \{\zeta\} \times \partial \Delta_{\sigma}^{n-q} \ \text{lies in}\ H^{q,n-q}.
\end{eqnarray}
We can shrink $H^{q,n-q}$ a little bit in order to get that $\Phi(H^{q,n-q})$ lies relatively compact in~$\Omega$. Let $\rho:\Omega \to \RR$ be an exhaustion function of~$\Mcal$. Then we can find a real constant $c$ such that $\Phi(H^{q,n-q})$ is contained in $\Omega_c:=\{z \in \Omega : \rho(z)< c\}$. In particular, $\rho \circ \Phi < c$ on $\partial A_{\zeta,\sigma}$ for every $|\zeta|<1$ and~$s < \sigma < 1$ by the inclusion in~(\ref{slices-in-omega}). Now the maximum principle for $q$-convex functions implies that $\rho \circ \Phi< c$ on $A_{\zeta,\sigma}$ for every $|\zeta|<1$ and $s < \sigma < 1$. By the identity in (\ref{slices-in-omega}), we conclude that $\rho \circ \Phi < c$ on $\Delta^n$, which yields $\Phi(\Delta^n) \subset \Omega_c \subset \Omega$. This means that $\Omega$ is $q$-pseudoconvex.
\end{proof}

We are now able to establish the first criteria for the existence of $q$-convex sets on compacts.

\begin{cor}\label{cor-exist-q-convex} Let $K$ be a compact $q$-pseudoconcave set in $\Mcal$. 

\begin{enumerate}

\item If there exists a $q$-convex function $\psi$ with corners in some open neighborhood of $K$, then $K$ is empty.

\item If $\Mcal$ is $q$-complete with corners, then $K$ is empty. In particular, if $\Mcal$ is Stein, then there are no compact $q$-pseudoconcave sets in~$\Mcal$.

\end{enumerate}

\end{cor}

\begin{proof} 1. Assume that there exists a function $\psi$ which is $q$-convex with corners on an open neighborhood of $K$ and that $K$ is non-empty. Then $\psi$ admits a maximum at a point $p_0$ in $K$, i.e.,
\[
\psi(p_0) = \max\{\psi(z) : z \in K\}.
\]
Choose a neighborhood $U$ of $p_0$ so small, that we can assume $U$ to be a local chart on $\Mcal$, and so that $\psi=\max\{\psi_1,\psi_2,\ldots,\psi_k\}$ on $U$, where $\psi_1,\psi_2,\ldots,\psi_k$ are $q$-convex functions on~$U$. Pick one of these functions, say $\psi_0:=\psi_{j_0}$ for some $j_0\in\{1,2,\ldots,k\}$, such that $\psi_{0}(p_0)=\psi(p_0)$. Then
\[
\psi_{0}(p_0) = \max\{\psi_{0}(z) : z \in U \cap K\}.
\]
Since $U$ is a local chart, and since $q$-convexity is invariant under biholomorphic change of coordinates, we can assume that $U$ lies inside the ball $\BB_R^n(0)$ with radius $R>0$ centered at the origin in $\CC^n$ and $p_0=0$. Notice also that $q$-pseudoconcavity is invariant under biholomorphic changes of coordinates. In particular, $U \cap K$ is $q$-pseudoconcave in $U$. Now let $B:=\BB_r^n(0) \Subset U$ be a smaller ball lying relatively compact in $U$ and choose an $\varepsilon_0>0$ such that $\Psi_0:=\psi_0-\eps_0\|z\|^2$ is still $q$-convex on $U$. Clearly, since $B$ contains $p_0=0$, we have
\[
\Psi_0(0)=\max\{\Psi_0(z) : z \in K \cap B\}.
\]
Then, by Proposition~\ref{prop-qcvx-loc-max}, we obtain
\[
\psi_0(0)=\Psi_0(0)\leq\max\{\Psi_0(z) : z \in K \cap B\} = \max\{\Psi_0(z) : z \in K \cap \partial B\}
\]
\[
 < \max\{\psi_0(z) : z \in K \cap \partial B\} = \max\{\psi_0(z) : z \in K \cap B\} = \psi_0(0).
\]
This leads to a contradiction on the assumption made on $K$ being nonempty. Therefore, $K$ must be empty, and we have proven the first part of this proposition.

\medskip

2. If $\Mcal$ is $q$-complete with corners, it admits a $q$-convex with corners exhaustion function. By the first part of this proposition, $K$ must be empty. If $\Mcal$ is Stein, then it admits a smooth strictly plurisubharmonic (i.e., $1$-convex) exhaustion function. Every $1$-convex function is automatically $q$-convex. Hence, any compact $q$-pseudoconcave set $K$ must be empty.

\end{proof}

\section{The \textit{q}-nucleus for compact sets}\label{sec-q-nuc-compact}

The $q$-pseudoconcavity gives already a condition on the existence of a $q$-convex function. The question remains whether it is possible to reverse statement~1 in Corollary~\ref{cor-exist-q-convex}. For this, we construct a special kind of subset of a given set $K$, the \emph{$q$-nucleus of $K$}, which is closely related on the existence of $q$-convex functions near $K$.

\begin{defn}\label{def-q-nucleus} Fix an integer number $q\in\{1,\ldots,n-1\}$. 

\begin{enumerate}

\item Let $K''\subset K'\subsetneq\Mcal$ be two proper compact sets in $\Mcal$. We say that $K''$ \emph{is obtained from $K'$ by a spherical cut of order $k$} if there exists a spherical hat pair $(S,\widehat{S})$ of order $k$ such that

\begin{itemize}

\item $S \subset \Mcal\setminus K'$,

\item and $K''=K' \setminus \Int(\widehat{S})$.

\end{itemize}


\item For two compacts $K''\subset K'$ we say that \emph{$K''$ is obtained from $K'$ by a sequence of spherical cuts of order $k$}, if there exists a finite sequence $\{K_j\}_{j=1}^m$ of compact sets such that

\begin{itemize}

\item $K_1 \supset K_2 \supset \cdots \supset K_m$,

\item $K_1=K'$ and $K''=K_m$,

\item and $K_{j+1}$ is obtained from $K_j$ by a spherical cut of order $k$ for each $j=1,\ldots,m-1$.

\end{itemize}

\item For a proper compact set $K \subsetneq \Mcal$ we define $\Fcal_K^q$ to be the set of all compacts $K''$ which are obtained from $K$ by a sequence of spherical cuts of order $n-q+1$. Then we define the \emph{$q$-nucleus $\nuc_q(K)$ of $K$} as the intersection of all compacts $K'' \in \Fcal_K^q$, i.e.,
\[
\nuc_q(K) := \bigcap_{K'' \in \Fcal_K^q} K''.
\]
\item Otherwise, if $K=\Mcal$, i.e., if $\Mcal$ is compact itself, we simply set $\nuc_q(\Mcal) :=\Mcal$.

\end{enumerate}

\begin{rem}\label{q-nucleus-compact} Notice that the $q$-nucleus is a compact set, and that the $1$-nucleus corresponds to the \emph{nucleus} introduced in~\cite{Shc21}.
\end{rem}

We need the following lemma to derive some important properties of the $q$-nucleus.

\begin{lem} For a finite family of sets $\{K''_j\}_{j=1}^l$ from $\Fcal^q_K$, the intersection $\D \bigcap_{j=1}^l K''_j$ also lies in $\Fcal^q_K$. 
\end{lem}

\begin{proof} This follows directly from the above definition. For details we refer to Lemma~3.1 in~\cite{Shc21}. Its proof can be transferred word by word to our setting.
\end{proof}

\end{defn}

Now we can give the relation of $q$-nucleus and $q$-pseudoconcavity.

\begin{prop}\label{prop-qnuc-qpsc} The $q$-nucleus of $K \subsetneq \Mcal$ is $q$-pseudoconcave in $\Mcal$. Moreover, it is the maximal $q$-pseudoconcave subset of $K$. In particular, if $K$ is $q$-pseudoconcave itself, then $\nuc_q(K)=K$. 
\end{prop}

\begin{proof} If the $q$-nucleus is empty, there is nothing to show, so we assume it is non-empty. Assume that $\nuc_q(K)$ is not $q$-pseudoconcave. 
According to Proposition~\ref{prop-qpsc-caps} there exists a spherical hat pair $(S,\widehat{S})$ of order $n{-}q{+}1$ such that

\begin{itemize}

\item $S \cap \nuc_q(K)=\emptyset$,

\item but $\widehat{S} \cap \nuc_q(K) \neq \emptyset$.

\end{itemize}

Pick a point $p \in \widehat{S} \cap \nuc_q(K)$. Then, in view of Lemma~3.3 and the definition of the $q$-nucleus, there exists a set $K'$ in $\Fcal_K^q$ such that

\begin{itemize}

\item $S \cap K'=\emptyset$,

\item but $\widehat{S} \cap K' \neq \emptyset$. More precisely, this intersection contains $p$.

\end{itemize}

We set $K'':= K' \setminus \Int(\widehat{S})$. It is obvious that $p \notin K''$ and that $K''$ is obtained from $K'$ by a spherical cut of order $n{-}q{+}1$. Since $K'$ is obtained from $K$ by a sequence of spherical cuts of order $n{-}q{+}1$, so is $K''$ as well. This means that $K''$ lies in $\Fcal_K^q$. Hence, by the definition of the $q$-nucleus and the choice of $p$, we have that $p \in \nuc_q(K) \subset K''$, so $p \in K''$, a contradiction. Therefore, $\nuc_q(K)$ is $q$-pseudoconcave.

Now assume that there exists a $q$-pseudoconcave set $A \subset K$. Let $K''\subsetneq\Mcal$ be a compact sets in $\Mcal$ such that $K''$ is obtained from $K$ by a spherical cut of order $n{-}q{+}1$, i.e., there exists a spherical hat pair $(S,\widehat{S})$ of order $n{-}q{+}1$ such that $S \subset \Mcal\setminus K$ and $K''=K \setminus \Int(\widehat{S})$. Since $S \subset \Mcal\setminus K \subset \Mcal\setminus A$ and $A$ is $q$-pseudoconcave, $\widehat{S} \cap A = \emptyset$ due to Proposition~\ref{prop-qpsc-caps}. Hence, $A$ lies in $K''$. Similarly, by an inductive argument, we can derive that if $K''$ is obtained from $K$ by a \emph{sequence} of spherical cuts of order $n{-}q{+}1$, then $A \subset K''$ as well. But this means that $A \subset \nuc_q(K)$ according to the definition of the $q$-nucleus of $K$. In this sense, the $q$-nucleus is maximal. It contains all $q$-pseudoconcave subsets of $K$. So if $K$ is $q$-pseudoconcave itself, then $K \subset \nuc_q(K) \subset K$. Thus, $\nuc_q(K)=K$.

\end{proof}

\begin{cor}\label{collect-qpsc-nucleus} Let $\Pcal_q(K)$ denote the collection of all $q$-pseudoconcave subsets of $K \subsetneq \Mcal$ and $\Pcal_q^*(K)$ the collection of all open $q$-pseudoconvex sets $U$ in $\Mcal$ such that $K \cap U = \emptyset$. Then
\[
\nuc_q(K) = \bigcup_{A \in \Pcal_q(K)} A,
\]
or, equivalently,
\[
\Mcal\setminus\nuc_q(K) = \bigcap_{U \in \Pcal_q^*(K)} U.
\]
In particular, $\Mcal\setminus\nuc_q(K)$ is $q$-pseudoconvex in $\Mcal$.
\end{cor}

The $q$-nucleus is a biholomorphic invariant in the following sense.

\begin{rem} 

Let $K$ be a compact set in $\Mcal$. For a given biholomorphic map $\Phi:\Mcal\to\Ncal$, where $\Ncal$ is another complex manifold, and for a compact set $K$ in $\Mcal$, we have that
\[
\Phi(\nuc_q(K)) = \nuc_q(\Phi(K)).
\]
This immediately follows from Corollary~\ref{collect-qpsc-nucleus} and from the fact that an open set $U$ in $\Mcal$ is $q$-pseudoconvex if and only $\Phi(U)$ is $q$-pseudoconvex in $\Phi(\Mcal)=\Ncal$.

\end{rem}

We give an example of the $q$-nucleus in the projective manifold.

\begin{ex}\label{q-nuc-cpn} Consider the embedded $\CC\PP^q$ in $\CC\PP^n$,
\[
\CC\PP^q=\{[z_0,\ldots,z_n] \in \CC\PP^n : z_{q+1}=\ldots=z_n=0\}.
\]
By computing the eigenvalues of
\[
\rho([z_0,\ldots,z_n])=\frac{|z_0|^2+\ldots+|z_q|^2}{|z_{q+1}|^2+\ldots+|z_n|^2}
\]
using local charts, we find that $\rho$ is an $(n-q)$-convex exhaustion function for $\CC\PP^n\setminus\CC\PP^q$. By Proposition~\ref{prop-q-cvx-psc}, the set $\CC\PP^q$ is a $q$-pseudoconcave compact subset of $\CC\PP^n$. Therefore, we can conclude that one has $\nuc_q(\CC\PP^q)=\CC\PP^q$ in $\CC\PP^n$. 
\end{ex}

Together with Corollary~\ref{cor-exist-q-convex} we obtain the following result. It states that the manifold cannot be too \emph{nice} in the geometric and holomorphic sense and that $K$ has to be large enough in order to produce a proper $q$-nucleus.

\begin{cor} \label{cor-empty-q-nucleus} The following two properties hold true:
\begin{enumerate}

\item If $\Mcal$ is $q$-complete (such as Stein in the case $q=1$), then $\nuc_q(K)$ is empty for any compact set $K$ in $\Mcal$. 

\item If $K$ is contained in a local holomorphic chart of $\Mcal$, then $\nuc_q(K)=\emptyset$.

\end{enumerate}
\end{cor}

Now we deal with the problem how to construct a $q$-convex function with corners in the neighborhood of $K$ in the case when the $q$-nucleus of $K$ is empty. First, we need the following lemma.

\begin{lem}\label{lem-glue-qcvx} Let $K$ be a compact set in $\Mcal$ and $V_1, V_2$ a pair of open sets in $\Mcal$ such that $K$ is contained in their union $V_1 \cup V_2$. Assume that there are two (weakly) $q$-convex functions with corners $\varphi_j$ defined on some neighbourhoods of $\overline{V_j}$ for $j=1,2$ such that $\varphi_1>\varphi_2$ on $\partial V_2 \cap V_1 \cap K$ and $\varphi_1<\varphi_2$ on $\partial V_1 \cap V_2 \cap K$. Define
\[
\Psi:=\left\{ \begin{array}{l} 
\varphi_1 \ \hbox{on}\ V_1\setminus V_2 \\
\varphi_2 \ \hbox{on}\ V_2\setminus V_1 \\
\max\{\varphi_1,\varphi_2\} \ \hbox{on}\ V_1 \cap V_2
\end{array}\right.
\]
Then there is a neighborhood $W$ of $K$ in $V_1 \cup V_2$ such that $\Psi$ is (weakly) $q$-convex with corners on~$W$.
\end{lem}

\begin{proof} It is obvious that $\max\{\varphi_1,\varphi_2\}$ is (weakly) $q$-convex with corners on $V_1\cap V_2$ simply by definition. We show that it extends according to the definition of $\Psi$ outside $K \cap V_1 \cap V_2$. For this, consider the compact sets $M_1:= \partial V_2 \cap V_1 \cap K$ and $M_2:= \partial V_1 \cap V_2 \cap K$. By the assumptions made on $\varphi_1$ and $\varphi_2$, there is an open neighborhood $U_1$ of $M_1$ in $V_1 \cup V_2$ such that $\varphi_1>\varphi_2$ on $U_1 \cap V_1$. Thus, $\max\{\varphi_1,\varphi_2\}=\varphi_1$ on $U_1 \cap V_1$ and easily extends by $\varphi_1$ into some open neighborhood $W_1$ of $K \setminus V_2$ to a $q$-convex function with corners. By the same arguments, we can extend $\max\{\varphi_1,\varphi_2\}$ by $\varphi_2$ into some open neighborhood $W_2$ of $K \setminus V_1$. Now denote by $\Psi$ the above constructed extension of $\max\{\varphi_1,\varphi_2\}$ into the open neighborhood $W=W_1 \cup W_2 \cup (V_1 \cap V_2)$ of $K$.
\end{proof}

\begin{prop}\label{hat-q-convex} 

Let $(S,\widehat{S})$ be a spherical hat pair of order $n-q+1$ in $\Mcal$. Then there is a neighborhood $U$ of the filled hat $\widehat{S}$ and a weakly $q$-convex function $\varphi$ on $U\setminus S$ such that $\varphi$ vanishes on $U\setminus \widehat{S}$ and $\varphi$ is positive and $q$-convex on the filling $\Int(\widehat{S})$.

\end{prop} 

\begin{proof} In the case $q=1$, the result is covered by Corollary~2.1 in \cite{Shc21}, so we continue with the case $q>1$. According to Theorem~2.1 of the same paper there exists a domain $W$ in $\CC^{n-q+1}$ and a smooth plurisubharmonic function $\tilde{\varphi}:W\to\RR$ such that
\begin{enumerate}
 
\item[i.] $\Pi_-:=\{z=(z_1,\ldots,z_{n-q+1}) \in \CC^{n-q+1} : \Repa{z_1}\leq 0\} \subset W$ and $\BB_1^{n-q+1}(0) \subset W$,

\item[ii.] $\tilde{\varphi}= 0$ on $\Pi_-$,

\item[iii.] $\tilde{\varphi}>0$ and strictly plurisubharmonic on $W\setminus\Pi_-$,

\end{enumerate}

Denote by $z'=(z_1,\ldots,z_{n-q+1}) \in \CC^{n-q+1}$, $z''=(z_{n-q+2},\ldots,z_n) \in \CC^{q-1}$ and $1'=(1,0,\ldots,0) \in \CC^{n-q+1}$. Thus, $z=(z',z'') \in \CC^n$. For $r>0$, we define the function $\psi_r$ as follows,
\[
\psi_r(z):= \left\{ \begin{array}{cl}(1-\|z''\|^2)\tilde{\varphi}\big( z'-r\cdot 1'\big), & \text{for}\ z \in \operatorname{Int}\big(\widehat{\SS}_r^{n-q+1}\times\Delta^{q-1}\big)\\
0, & \text{for}\ z \in \CC^n\setminus (\widehat{\SS}_r^{n-q+1}\times\Delta^{q-1}) \end{array} \right\}.
\]
Then $\psi_r$ has at least $n-q+1$ non-negative (in $z'$-direction) and, thus, at most $q-1$ negative eigenvalues. Hence, it is weakly $q$-convex everywhere where it is defined. Moreover, for every $0<r<1$ it fulfills

\begin{enumerate}

\item $\psi_r \equiv 0$ on $\CC^n\setminus (\widehat{\SS}_r^{n-q+1}\times\Delta^{q-1})$.

\item $\psi_r$ is positive and $q$-convex on the interior of $\widehat{\SS}_r^{n-q+1}\times\Delta^{q-1}$.

\end{enumerate}

Now we simply set $\varphi:=\psi_r \circ \Phi^{-1}$ for the proper choice of $r$, where $\Phi$ is the function from Definition~\ref{defn-qpsc-hats}. It is easy to see that $\varphi$ fulfills all the properties posed in the statement.

\end{proof}

We are now able to formulate and prove the central result of this paper which is given by the following statement. 

\begin{mthmn}\label{main-theorem} Let $K$ be a compact set in $\Mcal$.

\begin{enumerate}

\item If there is a $q$-convex function with corners defined on a neighborhood of $K$, then $\nuc_q(K)$ is empty. In particular, $K$ does not contain any $q$-pseudoconcave subset. 

\item If $\nuc_q(K)$ is empty, then there exists a (positive) $q$-convex function with corners defined on a neighborhood of $K$. 

\end{enumerate}

\end{mthmn}

\begin{proof} 1. Let $\psi$ be a $q$-convex function with corners on a neighborhood of $K$. Notice that $K \subsetneq \Mcal$ due to point~3 in Remark~\ref{rem-qvx-fcts}. Assume that $\nuc_q(K)$ is not empty. Since $\nuc_q(K)$ lies in $K$, the function $\psi$ is also defined in a neighborhood of $\nuc_q(K)$. 
Since the $q$-nucleus is compact and $q$-pseudoconcave, the first part of Corollary \ref{cor-exist-q-convex} applied to $\nuc_q(K)$ (instead of $K$ there) will lead us to a contradiction. This, together with Corollary~\ref{collect-qpsc-nucleus}, proves the first part of our Main Theorem.

\medskip

2. We are able to follow mainly the arguments in the second part of the proof of the Main Theorem in~\cite{Shc21} except for the fact that we have to take maximums instead of sums in the last step. For the readers convenience we shall give a detailed proof here.

Since $\nuc_q(K)$ is empty, we can find a sequence $K_1 \supset K_2 \supset \ldots \supset K_{m-1} \supset K_m$ such that $K_1=K$, $K_m=\emptyset$ and $K_{j} $ is obtained from $K_{j-1}$ by a spherical cut of order $n{-}q{+}1$. We proceed inductively in order to construct for each $j=m-1, m-2,\ldots,1$ an open neighborhood $U_{j}$ of $K_j$ and a positive $q$-convex function $\Psi_j$ with corners  on $U_j$.

We start with the base case. We know that $K_{m}$ is obtained from $K_{m-1}$ by a spherical cut of order $n-q+1$, i.e., there exists a spherical hat pair $(S,\widehat{S})$ of order $n-q+1$ such that $S \subset \Mcal \setminus K_{m-1}$ and $K_{m} = K_{m-1}\setminus \Int(\widehat{S})$ . By Proposition~\ref{hat-q-convex} there exist a neighborhood $U$ of the filled hat $\widehat{S}$ and a weakly $q$-convex function $\Psi_{m-1}:=\varphi_{m-1}$ on $U\setminus S$ such that $\varphi_{m-1}$ vanishes on $U\setminus \widehat{S}$ and is positive and $q$-convex on $U_{m-1}:=\Int(\widehat{S})$. 

We proceed with the induction step and assume that for some $j \in \{2, 3 , \ldots , m - 1\}$ we have already constructed an open neighborhood $U_j$ of $K_j$ and a positive $q$-convex function $\Psi_j$ with corners on $U_j$. We set $V_1:=U_j$. Since $K_j$ is obtained from $K_{j-1}$ by a spherical cut of order $n-q+1$, there exists a spherical hat pair $(S,\widehat{S})$ of order $n-q+1$ such that $S \subset \Mcal \setminus K_{j-1}$ and $K_{j} = K_{j-1}\setminus \Int(\widehat{S})$. 

By Proposition~\ref{hat-q-convex} there exist a neighborhood $U$ of the filled hat $\widehat{S}$ and a weakly $q$-convex function $\Psi_{j-1}:=\varphi_{j-1}$ on $U\setminus S$ such that $\varphi_{j-1}$ vanishes on $U\setminus \widehat{S}$ and is positive and $q$-convex on the filling $V_2:=\Int(\widehat{S})$. Observe that $K_{j-1}$ is contained in $V_1 \cup V_2$. 

Since $\varphi_{j-1}$ vanishes on $\partial V_2 \cap K_{j-1}$, and since $\Psi_j$ is positive, we conclude that $\Psi_j > c\varphi_{j-1}$ on $ \partial V_2 \cap V_1 \cap K_{j-1}$ for every positive constant $c>0$. On the other hand, since $\varphi_{j-1}$ positive on $V_2$, we can find a constant $c>0$ so large, that $\Psi_j < c\varphi_{j-1}$ on $ \partial V_1 \cap V_2 \cap K_{j-1}$. Notice that we can shrink $V_1$ and $V_2$ so that $V_1$ and $V_2$ are bounded and $\Psi_j$ and $\varphi_j$ are continuous up to the boundary of $V_1$ and $V_2$. 

By Lemma~\ref{lem-glue-qcvx}, we are now able to glue $\Psi_j$ and $c\varphi_{j-1}$ to a positive $q$-convex function with corners $\Psi_{j-1}$ on some neighborhood $U_{j-1}$ of $K_{j-1}$.

Proceeding in such a manner inductively, we can construct a function $\Psi_1$ which is $q$-convex with corners on some neighborhood $U_1$ of $K_1=K$, which completes the proof of our Main Theorem.

\end{proof}

The Main Theorem implies conditions on the existence of the $q$-nucleus in $q$-convex manifolds.

\begin{rem} If $\Mcal$ is a \emph{$q$-convex manifold}, i.e., if there exists a $q$-convex exhaustion function defined on $\Mcal\setminus L$ for some compact set $L$ in $\Mcal$, then the Main Theorem asserts that $\nuc_q(K)=\emptyset$ for every compact set $K$ in $\Mcal\setminus L$. Therefore, if there exists a compact $K$ in $\Mcal$ such that $\nuc_q(K)\neq \emptyset$, then necessarily $K$ intersects the boundary of $\Mcal\setminus L$ in $\Mcal$. Furthermore, since $\Omega:=\Mcal\setminus L$ is $q$-complete, we have that $L$ itself is $q$-pseudoconcave, and then $\nuc_q(L)=L$ by Proposition~\ref{prop-qnuc-qpsc}.
\end{rem}

\section{The \textit{q}-nucleus for closed sets}

Before we define the $q$-nucleus for closed sets, we need the following monotonicity property of the $q$-nucleus for compact sets. From now on, we always assume that $q\in\{1,\ldots,n-1\}$. 

\begin{lem}\label{lem-monoton-qnuc} Given two compact sets $K$, $L$ in $\Mcal$ with $K \subset L$, we have $\nuc_q(K) \subset \nuc_q(L)$.
\end{lem}

\begin{proof} According to Proposition~\ref{prop-qnuc-qpsc}, the $q$-nucleus of $K$ is a $q$-pseudoconcave set in $K \subset L$. Since, due to Proposition~\ref{collect-qpsc-nucleus}, $\nuc_q(L)$ is maximal and contains all $q$-pseudoconcave sets in $L$, we derive that $\nuc_q(K) \subset \nuc_q(L)$.
\end{proof}

Assume from now on that $\Mcal$ is locally compact. Let $A$ be a closed set in $\Mcal$. Let $\{K_k\}_{k \in \NN}$ and $\{L_m\}_{m \in \NN}$ be sequences of compact sets in $\Mcal$ with $K_k \subset K_{k+1}$ for every $k \in \NN$ and $\Mcal=\bigcup_k K_k$ and $L_m \subset L_{m+1}$ for every $m \in \NN$ and $\Mcal=\bigcup_m L_m$. Then for each $\nu \in \NN$ there exist $m_\nu, k_\nu \in \NN$ such that $K_\nu \subset L_{m_\nu}$ and $L_\nu \subset K_{k_\nu}$. By the monotonicity of the $q$-nucleus (Lemma~\ref{lem-monoton-qnuc}), we obtain
\[
\nuc_q(A \cap K_{\nu}) \subset \nuc_q(A \cap L_{m_\nu}) \qand \nuc_q(A \cap L_\nu) \subset \nuc_q(A \cap K_{k_\nu}).
\]
It follows that
\[
\overline{ \bigcup_{\nu \in \NN} \nuc_q(A \cap K_{\nu})} \ \subset\ \overline{ \bigcup_{\nu \in \NN} \nuc_q(A \cap L_{m_\nu})} \ \subset\ \overline{ \bigcup_{\nu \in \NN} \nuc_q(A \cap L_{\nu})}
\]
\[
 \subset\ \overline{ \bigcup_{\nu \in \NN} \nuc_q(A \cap K_{k_\nu})} \ \subset\ \overline{ \bigcup_{\nu \in \NN} \nuc_q(A \cap K_{\nu})}.
\]
From this, we derive the identity $\D \overline{ \bigcup_{\nu \in \NN} \nuc_q(A \cap K_{\nu})} = \overline{ \bigcup_{\nu \in \NN} \nuc_q(A \cap L_{\nu})}$, i.e. the set $\D \overline{ \bigcup_{\nu \in \NN} \nuc_q(A \cap K_{\nu})}$ does not depend on the choice of the exhaustion  $\{K_\nu\}_{\nu \in \NN}$ of $\Mcal$.

\begin{defn}\label{defn-closed-q-nucleus} Let $\{K_n\}_{n \in \NN}$ be an arbitrary sequence of compact sets $\{K_n\}_{n \in \NN}$ such that $K_k \subset K_{k+1}$ for every $n \in \NN$ and $\Mcal = \bigcup_{k} K_k$. Then we define the \emph{$q$-nucleus for a closed set $A$ in $\Mcal$} by
\[
\nuc_q(A):= \overline{ \bigcup_{k \in \NN} \nuc_q(A \cap K_k)}.
\]
\end{defn}

We give a necessary condition on the existence of $q$-convex functions near closed sets.

\begin{prop}\label{closed} Let $A$ be a closed set in $\Mcal$. If there exists a $q$-convex function with corners in an open neighborhood $U$ on $A$, then $\nuc_q(A)$ is empty.
\end{prop}

\begin{proof} If there exists a $q$-convex function with corners on a neighborhood $U$ of $A$, then $\nuc_q(A \cap K)$ is empty for every compact set $K \subset \Mcal$ according to the Main Theorem of this paper. Thus, the $q$-nucleus of $A$ has to be empty.
\end{proof}

We present a rather obvious example in the Stein case.

\begin{ex} If $\Mcal$ is Stein and $A$ is a properly embedded complex submanifold of $\Mcal$, then, since a Stein manifold always admits a smooth strictly plurisubharmonic (i.e. 1-convex) exhaustion function, we can apply Proposition 4.3 with $\Mcal$ on the place of $U$ to conclude that the $q$-nucleus of $A$ is empty. Recall here, that 1-convex function is automatically $q$-convex.
\end{ex}

The next observation follows directly from the definition.

\begin{prop} The $q$-nucleus of a closed set $A \subsetneq \Mcal$ is $q$-pseudoconcave in $\Mcal$.
\end{prop}

\begin{proof} Since $\nuc_q(A)= \hbox{closure\ of}\ \bigcup_{n \in \NN} \nuc_q(A \cap K_n)$ and the closure of the union of arbitrary many $q$-pseudoconcave sets is again $q$-pseudoconcave, the $q$-nucleus of $A$ is $q$-pseudoconcave in $A$. 
\end{proof}

An analog of Corollary~\ref{collect-qpsc-nucleus} does not hold for the $q$-nucleus of closed, not necessarily compact sets.

\begin{rem}\label{ex-closed-q-nuc-1} The $q$-nucleus of a closed set $A$ is not necessarily the maximal $q$-pseudoconcave subset of $A$. Indeed, if it would be maximal, then
\[
\bigcup_{P \in \Pcal_q(A)} P = \nuc_q(A),
\]
where $\Pcal_q(A)$ denotes the collection of all $q$-pseudoconcave subsets of $A$. But if we consider the case $A=\CC^n$ in $\Mcal=\CC^n$, the set $A$ can be filled by complex submanifolds of dimension~$q$. Such manifolds are $q$-pseudoconcave  (see Example~\ref{ex-q-pscve-analytic}). Hence, $\bigcup_{P \in \Pcal_q(\Mcal)} P = \CC^n$, whereas $\nuc_q(\CC^n)=\emptyset$ due to the existence of a strictly plurisubmarmonic function on the whole of $\CC^n$, e.g., $\psi(z)=\|z\|^2$.

In fact, the $q$-nucleus of $A$ is maximal with respect to \emph{compact} $q$-pseudoconcave subset of $A$. So if we denote by $\Pcal^c_q(\Mcal)$ the collection of all compact $q$-pseudoconcave subsets of $A$, we have
\[
\bigcup_{A \in \Pcal^c_q(\Mcal)} A = \nuc_q(A).
\]
\end{rem}


We end this section with the following remark.

\begin{rem} If $S:=\nuc_{n-1}(\Mcal)$ is neither empty nor the whole of $\Mcal$, it is a proper $(n-1)$-pseudoconcave set in $\Mcal$. Since $(n-1)$-pseudoconcave means pseudoconcave, the set $\Mcal\setminus S$ is pseudoconvex, and in this sense, $S$ can be considered as an analytic multifunction. 
\end{rem}

\section{Open questions}

The second part of our Main Theorem asserts that if the nucleus $\nuc_q(K)$ of a compact set $K \subset \Mcal$ is empty, then there is a $q$-convex function with corners defined on a neighborhood of $K$. In the case $q=1$ the main result in~\cite{Shc21} gives a better statement, namely, that if $\nuc_1(K) = \emptyset$, then there is a smooth $1$-convex function  in a neighbourhood of $K$. It is well known, due to the results in ~\cite{DF}, that for $q > 1$, in general, functions with corners can not be approximated by smooth $q$-convex functions. This does not mean, however, that in our case such smooth $q$-convex function (maybe obtained by a completely different method) does not exist. This motivates our first question.

\begin{quest1} Let $\Mcal$ be a complex manifold and $K \subset \Mcal$ be a compact set such that for some   $q > 1$ one has $\nuc_q(K) = \emptyset$. Is it always true that there exists a \emph{smooth} $q$-convex function defined in an open neighborhood of $K$? 
\end{quest1} 

Our next question is concerned with the $q$-nucleus for closed sets and it asks if the reverse to the statement of Proposition~\ref{closed} is true or not.

\begin{quest2} Let $q\in\{1,\ldots,n-1\}$ and let $A$ be a closed set in $\Mcal$. Assume that $\nuc_q(A) = \emptyset$. Is it always true that there exists a $q$-convex function with corners defined in an open neighborhood of $A$? 
\end{quest2}

\vspace{3mm}
\noindent
{\bf Acknowledgement.} {\it Part of this work was done while the second author was a visitor at the Southern University of Science and Technology (Shenzhen). It is his pleasure to thank this institution for its hospitality and good working conditions.}

\bibliographystyle{alpha}

%
%
%
%
%
%
\vskip1,2cm  
{\sc T. Pawlaschyk: Department of Mathematics, University of Wuppertal --- 42119 Wuppertal, Germany}
  
{\em e-mail address}: {\texttt pawlaschyk@uni-wuppertal.de}

\vskip0,6cm   
{\sc N. Shcherbina: Department of Mathematics, University of Wuppertal --- 42119 Wuppertal, Germany}
  
{\em e-mail address}: {\texttt shcherbina@math.uni-wuppertal.de}

\end{document}